\documentclass[11pt,a4paper,reqno]{amsart}%
\usepackage{amsbsy,amsopn,amscd,amstext,amsgen,amsthm,amsmath,amsfonts,amssymb,amsxtra,appendix,bookmark,dsfont,latexsym,color,epsfig,xspace,mathrsfs,dsfont,srcltx}

\usepackage{hyperref}
\usepackage{delarray,a4,color}
\usepackage{euscript}
\usepackage[latin1]{inputenc}

\makeatletter
\def\@tocline#1#2#3#4#5#6#7{\relax
  \ifnum #1>\c@tocdepth 
  \else
    \par \addpenalty\@secpenalty\addvspace{#2}%
    \begingroup \hyphenpenalty\@M
    \@ifempty{#4}{%
      \@tempdima\csname r@tocindent\number#1\endcsname\relax
    }{%
      \@tempdima#4\relax
    }%
    \parindent\z@ \leftskip#3\relax \advance\leftskip\@tempdima\relax
    \rightskip\@pnumwidth plus4em \parfillskip-\@pnumwidth
    #5\leavevmode\hskip-\@tempdima
      \ifcase #1
       \or\or \hskip 1em \or \hskip 2em \else \hskip 3em \fi%
      #6\nobreak\relax
      \dotfill
      \hbox to\@pnumwidth{\@tocpagenum{#7}}
    \par
    \nobreak
    \endgroup
  \fi}
\makeatother

\theoremstyle{plain}
\newtheorem{theorem}{Theorem}[section]
\newtheorem{lemma}[theorem]{Lemma}

\newtheorem{proposition}[theorem]{Proposition}

\newtheorem{asumption}[theorem]{Asumption}
\newtheorem{definition}[theorem]{Definition}

\theoremstyle{remark}
\newtheorem{remark}[theorem]{Remark}

\numberwithin{equation}{section}

\DeclareMathOperator{\tr}{Tr}

\def\geqslant{\ge}
\def\leqslant{\le}
\def\bq{\begin{eqnarray}}
\def\eq{\end{eqnarray}}
\def\bqq{\begin{eqnarray*}}
\def\eqq{\end{eqnarray*}}

\def\eps{\varepsilon}
\def\wto{\rightharpoonup}

\newcommand{\norm}[1]{\left\lVert #1 \right\rVert}

\newcommand\1{{\ensuremath {\mathds 1} }}

\renewcommand{\epsilon}{\varepsilon}

\def\gF {\mathfrak{F}}

\def\cH{\mathcal{H}}

\def\R {\mathbb{R}}

\def\C {\mathbb{C}}

\def\cS {\mathcal{S}}

\def\cP {\mathcal{P}}

\def\cE {\mathcal{E}}

\def\R {\mathbb{R}}
\def\C {\mathbb{C}}

\def\gS{\mathfrak{S}}

\def\cS {\mathcal{S}}

\def\cM {\mathcal{M}}

\def\gH{\mathfrak{H}}

\renewcommand{\leq}{\leqslant}
\renewcommand{\geq}{\geqslant}

\newcommand{\sym}{{\mathrm{sym}}}

\newcommand{\im}{{\rm{i}}}

\newcommand{\be}{\begin{equation}}
\newcommand{\ee}{\end{equation}}
\newcommand{\bea}{\begin{eqnarray}}
\newcommand{\eea}{\end{eqnarray}}

\def\proba{{\rm I\kern -.18em P}}

\newcommand{\Gammat}{\tilde{\Gamma}}
\newcommand{\bL}{\mathbf{L}}

\newcommand{\cEnls}{\mathcal{E} ^{\rm nls}}
\newcommand{\Enls}{E ^{\rm nls}}

\newcommand{\Mnls}{\mathcal{M} ^{\rm nls}}

\newcommand{\cEnlsm}{\mathcal{E} ^{\rm nls,\mathrm{m}}}
\newcommand{\Enlsm}{E ^{\rm nls,\mathrm{m}}}

\newcommand{\Mnlsm}{\mathcal{M} ^{\rm nls,\mathrm{m}}}

\newcommand{\cEH}{\mathcal{E} ^{\rm H}}
\newcommand{\EH}{E ^{\rm H}}

\title[]{Non linear Schr\"odinger limit of bosonic ground states, again} 

\author[N. Rougerie]{Nicolas ROUGERIE}
\address{Universit\'e Grenoble Alpes \& CNRS,  LPMMC,  F-38000 Grenoble, France}
\email{nicolas.rougerie@lpmmc.cnrs.fr}

\date{February, 2020}

\begin{document}

\begin{abstract}
I review an information-theoretic variant of the quantum de Finetti theorem due to Brand\~{a}o and Harrow and discuss its applications to the topic of bosonic mean-field limits. This leads to slightly improved methods for the derivation of the local non-linear Schr\"odinger energy functional from many-body quantum mechanics.
\end{abstract}

\maketitle

\setcounter{tocdepth}{2}
\tableofcontents

\section{Introduction}\label{sec:intro}

The present paper is an addendum to the ``quantum de Finetti-based'' approach to the mean-field limit of bosonic ground states developed over the past few years. It has two main motivations:

\smallskip

\noindent \textbf{1.} Review an interesting variant of the quantum de Finetti theorem derived in~\cite{BraHar-12,LiSmi-15}.

\smallskip
 
\noindent\textbf{2.} Couple this variant to the approach of bosonic mean-field limits described in~\cite{LewNamRou-14c}, in order to remove unaesthetic restrictions on its range of validity. 

\smallskip

Both motivations are mostly pedagogical. Let me briefly discuss the second one. 

We consider the ground state(s) of a many-body Hamiltonian of the form 
\begin{equation}\label{eq:start hamil}
H_N = \sum_{j=1} ^N \left( \left( -\im \nabla_{x_j} + A (x_j) \right) ^2 + V (x_j)\right) + \frac{1}{N-1} \sum_{1 \leq j < k \leq N} N ^{d\beta} w (N ^{\beta} (x_j-x_k)) 
\end{equation}
acting on $L^2_{\rm sym} (\R ^{dN})$, the space of symmetric $N$-body wave-functions appropriate for the description of bosonic particles. Here $d = 1,2,3$ is the dimension of the physical space, $V:\R^{d} \mapsto \R$ is an external potential, $A:\R^d \mapsto \R^d$ the vector potential of an external magnetic field $B= \mathrm{curl} \, A$ and $w:\R^d \mapsto \R$ a pair interaction potential. Our convention is that the length scale of the system is set by the external potential $V$, that we shall take trapping:
\begin{equation}\label{eq:trap s}
V (x) \geq c |x| ^s - C \mbox{ for some fixed constants } c,C,s>0.
\end{equation}
The scaling of the interactions is then designed to impose 
\begin{equation}\label{eq:basic scaling}
\mbox{(range of interactions)} ^d \times \mbox{average density} \times \mbox{interaction strength} = O (1) 
\end{equation}
in the limit $N\to \infty$, so that the interaction energy is of the same magnitude as typical one-particle energies. Fixing~\eqref{eq:basic scaling} still leaves some freedom, and the fixed parameter $\beta \geq 0$ is used to interpolate between two scenarios:

\smallskip 

\noindent \textbf{1.} $\beta < 1/d$ is a mean-field regime, interactions are of longer range than the mean inter-particle distance.

\smallskip 

\noindent \textbf{2.} $\beta > 1/d$ is a dilute regime, interactions are of shorter range than the mean inter-particle distance.

\smallskip

Maximum physical relevance demands rather large values of $\beta$: in 3D $\beta = 1$ is, for reasons explained at length elsewhere, e.g. in~\cite[Chapters~2 and~6]{LieSeiSolYng-05} or~\cite[Chapter~7]{Rougerie-LMU,Rougerie-spartacus}, the most relevant case. In 2D one might even consider an exponential-like scaling of the interactions' range (thus $\beta = \infty$ formally), see~\cite[Chapters~3 and~6]{LieSeiSolYng-05}. See~\cite{Rougerie-EMS} for a general perspective on scaling limits of bosonic ground states.

However, the larger $\beta$, the harder the analysis. For small values of $\beta$ it is feasible to deal with the $N\to \infty$ limit of the ground-state problem 
\begin{equation}\label{eq:GSE}
E(N) := \min \left\{ \langle \Psi_N | H_N | \Psi_N \rangle, \: \Psi_N \in L^2_{\rm sym} (\R^{dN}) \right\}  
\end{equation}
using only general structural facts of many-body quantum mechanics~\cite{LewNamRou-14,LewNamRou-14c}, within a totally variational\footnote{In the spirit of $\Gamma$-convergence one might say.} proof. 

This paper is concerned with improving the conditions on $\beta$ (i.e. the rate at which the interactions converge to point-like ones) under which one can treat the $N\to \infty$ limit of Problem~\eqref{eq:GSE} in a totally variational way. We are able to handle the $N\to \infty$ limit of Problem~\eqref{eq:GSE} provided one stays reasonably deep within the mean-field regime:
\begin{equation}\label{eq:intro beta}
\beta < \frac{1}{2d}. 
\end{equation}
Our general approach to the mean-field limit is that of~\cite{LewNamRou-14c,LunRou-15}, but we use as main tools the results of~\cite{BraHar-12,LiSmi-15} instead of those of~\cite{ChrKonMitRen-07,LewNamRou-14b}. Using the latter in~\cite{LewNamRou-14c} led to the condition $\beta < \beta_0 (s)$ for some rather small and $s$-dependent $\beta_0(s)$ ($s$ being the exponent in~\eqref{eq:trap s}). This annoying dependence on $s$ gets dispensed with here.

\medskip

The value~\eqref{eq:intro beta} is optimal with the method exposed here. Larger values of $\beta$ are known to be reachable by methods outside of the range of this paper (see the comments after Theorem~\ref{thm:main}). The complexity of the proofs increases rather steeply~\cite{BocBreCenSch-18b,LieYng-98,LieYng-01,LieSeiYng-00,LieSeiYng-01,LieSei-02} for $\beta > 1/d$. The proofs are no longer purely variational, as one typically uses the many-body Schr\"odinger equation to obtain a priori bounds on minimizers~\cite{LewNamRou-15,LieSei-06,NamRouSei-15}. I also mention that similar problems and techniques are useful in the context of the ``almost bosonic anyon gas''~\cite{LunRou-15,Girardot-19} and the dipolar Bose gas~\cite{Triay-17}. See also~\cite{BarGolMau-00,AmmNie-08,FroKnoSch-09,RodSch-09,Pickl-11,ErdSchYau-07,Pickl-10,ErdSchYau-09,Pickl-11,BenOliSch-12,NamNap-17,CheHol-13,CheHol-15,JebLeoPic-16} for a selection of works dealing with mean-field and/or dilute limits of the dynamical problem associated with~\eqref{eq:start hamil}.

\medskip

\noindent \textbf{Organization of the paper.} In Section~\ref{sec:main} I state the main result on the mean-field limit of~\eqref{eq:start hamil}. Section~\ref{sec:proof} explains the adaptations to be made in the proof strategy of~\cite{LewNamRou-14c}. The main one is to use an information-theoretic quantum de Finetti theorem. Its statement and proof are reviewed in Appendix~\ref{sec:BraHar} for the benefit of readers who, like myself, lack familiarity with arguments that are standard in quantum information theory, but much less so in many-body quantum mechanics.

\bigskip

\noindent \textbf{Notation.} For a vector $\psi$ in a Hilbert space $\gH$ (usually a function $\psi \in L ^2 (\R^d)$), we use the bra-ket $|\psi \rangle \langle \psi|$ notation for the corresponding orthogonal projector (pure state). 

The symbol $\tr$ stands for the trace. When decorated with subscripts, a partial trace with respect to these subscripts is meant. That is, for an operator acting on $\gH_1 \otimes \ldots \otimes \gH _N$, $\tr_{i_1,\ldots,i_k}$ means tracing over $\gH_{i_1},\ldots,\gH_{i_k}$ If the operator one takes the trace of acts on a tensor product $\gH^{\otimes N}$ and is symmetric, I  indicate $\tr_{k+1 \to N}$  to mean a partial trace with respect to $N-k$ factors of the $N$-fold tensor product, no matter which.

The interaction potential in~\eqref{eq:start hamil} is denoted 
\begin{equation}\label{eq:w N beta}
w_{N,\beta} (x) := N ^{d\beta} w(N ^{\beta} x ). 
\end{equation}

\bigskip

\noindent \textbf{Acknowledgments.} Thanks to Isaac Kim who first drew my attention to~\cite{BraHar-12} some years ago. Funding from the European Research Council (ERC) under the European Union's Horizon 2020 Research and Innovation Programme (Grant agreement CORFRONMAT No 758620) is gratefully acknowledged.

\section{Results and discussion}\label{sec:main}

\subsection{Statements} For the easiest case $d=1$,~\cite{LewNamRou-14c} already covers any value of $\beta >0$ with a fully variational method. We thus focus on the cases $d=2,3$ and work under the following assumptions:

\begin{asumption}[\textbf{Stable interactions}]\label{asum:stable}\mbox{}\\
The pair interaction potential $w:\R^d \mapsto \R$ is a bounded, integrable even function with $|x| w (x) \in L^1 (\R^d)$ and Fourier transform $\widehat{w} \in L^1 (\R^d)$. It is stable in the sense that 

\smallskip

\noindent $\bullet$ in 3D, $w\geq 0$.

\smallskip

\noindent $\bullet$ in 2D, $\int_{\R^d} |w_-| < a_*$ with $w_-$ the negative part of $w$ and $a_*$ the optimal constant in the Gagliardo-Nirenberg inequality, 
$$ \int_{\R^2} |u| ^4 \leq a_* \left( \int_{\R^2} |u| ^2 \right) \left( \int_{\R^2} |\nabla u| ^2 \right)$$
Equivalently $a_* = \norm{Q}_{L^2 (\R^2)} ^2$ where $Q$ is the unique (up to translations) solution to
$$ -\Delta Q + Q - Q^3 = 0 \mbox{ on } \R^2.$$
\end{asumption}

See~\cite{LewNamRou-14c} or~\cite[Chapter~7]{Rougerie-LMU} for more comments on the necessity of these assumptions. In 2D one can relax the condition to what we called ``Hartree-stability'' in~\cite{LewNamRou-14c}, but that is a small improvement that I sacrifice here for simplicity. For the other data of the problem we make standard assumptions:
$$ A\in L^2_{\rm loc} (\R^d,\R^d), V \in L^1 _{\rm loc} (\R^d,\R)$$
and assume~\eqref{eq:trap s}. 

The limiting objects in the $N\to \infty$ limit are as follows. Let the non-linear Schr\"odinger (NLS) functional be
\begin{equation}\label{eq:nls func}
\cEnls [u] := \int_{\R^d} \left| (-\im \nabla + A) u \right| ^2 + V|u| ^2 + \frac{a}{2} \int_{\R^d} |u| ^4
\end{equation}
with 
$$ a = \int_{\R ^d} w.$$
We also define the associated ground-state energy
\begin{equation}\label{eq:nls ener}
\Enls = \min \left\{ \cEnls [u], \: \int_{\R^d} |u|^2 = 1 \right\}
\end{equation}
with $\Mnls$ the set of associated minimizers. Our main result is

\begin{theorem}[\textbf{Mean-field/NLS limit of bosonic ground states}]\label{thm:main}\mbox{}\\
Let 
$$ 0 < \beta < \frac{1}{2d}$$
be a fixed parameter. The following holds in the $N\to \infty$ limit. 

\smallskip 

\noindent \textbf{1.} Convergence of the $N$-body ground state energy
\begin{equation}\label{eq:resul ener}
\frac{E(N)}{N} \underset{N\to \infty}{\longrightarrow} \Enls.  
\end{equation}

\smallskip 

\noindent \textbf{2.} Convergence of reduced density matrices. Let $(\Psi_N)_N$ be a sequence of quasi minimizers for~\eqref{eq:GSE}, i.e.
\begin{equation}\label{eq:quasi min}
\langle \Psi_N | H_N | \Psi_N \rangle \leq E (N) + o (N). 
\end{equation}
For $k\geq 1$ let $\gamma_N ^{(k)}$ be the associated $k$-particles reduced density matrix\footnote{$\tr_{k+1\to N}$ means partial trace on $N-k$ particles, cf the Notation paragraph at the end of Section~\ref{sec:intro}.}
\begin{equation}\label{eq:reduced DM}
 \gamma_N ^{(k)} = \tr_{k+1 \to N} |\Psi_N \rangle \langle \Psi_N |.
\end{equation}
Then 
\begin{equation}\label{eq:result DM}
 \gamma_N ^{(k)} \underset{N\to \infty}{\longrightarrow} \int | u^{\otimes k} \rangle \langle u ^{\otimes k}| d\mu(u)
\end{equation}
strongly in the trace-class, where $\mu$ is a Borel probability measure supported on $\Mnls$.
\end{theorem}

Note that 

\smallskip 

\noindent \textbf{1.} Under the stated assumptions it is standard to see that both~\eqref{eq:GSE} and~\eqref{eq:nls ener} are well-posed.

\smallskip 

\noindent \textbf{2.} The result is not new, and is actually weaker than what was known already using more-than-variational proofs. In 3D any $\beta < 1$ can be covered with the methods of~\cite{LieSei-06,NamRouSei-15}. Indeed, these papers deal with the special, harder case $\beta = 1$ (where the result is not stated the same). In 2D~\cite{LewNamRou-15} can handle any $\beta < (s+1)/(s+2)$ (in particular some dilute regimes $\beta >1/2$ are allowed). The novel aspect is thus methodological, as we discuss next. 

\smallskip 

\noindent \textbf{3.} Coupling the method of the present paper with refined a-priori estimates allows~\cite{NamRou-19} to reach any value of $\beta < 1$ in 2D. 

\smallskip

\subsection{Method of proof} We pursue along the lines of~\cite{LewNamRou-14c}. The argument is based on the quantum de Finetti theorem, which asserts that~\eqref{eq:result DM} holds for essentially any sequence of bosonic states $\Psi_N$, for some general Borel probability measure $\mu$ over $L^2 (\R ^d)$ (not necessarily supported on $\Mnls$). The goal is then to identify the support of the measure associated with sequences of quasi-minimizers. 

The difficulty in applying this general idea to NLS-like limits ($\beta > 0$) is that one cannot use soft compactness arguments to pass to the limit in the energy. The idea of~\cite{LewNamRou-14c} is to rely on specific versions of the quantum de Finetti theorem which explicitly quantify the error made in replacing the left side of~\eqref{eq:result DM} by the right side. Unfortunately, explicit estimates are available only when the one-body Hilbert space $L^2 (\R^d)$ is replaced by a finite-dimensional one. Thus the need to 
\begin{itemize}
 \item project the problem to finitely many dimensions, i.e. on one-body states whose one-body energy is below a certain energy cut-off.
 \item use the quantitative finite dimensional de Finetti theorem in the projected space.
 \item take the energy cut-off high enough for states in the orthogonal complement to be  negligible.
\end{itemize}
The technical limitations imposed on $\beta$ in~\cite{LewNamRou-14c} arose because the de Finetti theorem~\cite{ChrKonMitRen-07,Chiribella-11,Harrow-13,LewNamRou-14b}  we used had errors depending linearly on the low-energy space's dimension. The latter depends polynomially on the energy cut-off (this can be seen by Cwikel-Lieb-Rosenblum-type bounds). We here relax these limitations by using a finite dimensional quantitative de Finetti theorem whose errors~\cite{BraHar-12} depend only logarithmically\footnote{Thus, for many practical purposes, the theorem is almost as good as a quantitative de Finetti theorem in infinite dimension.} on the dimension of the one-body Hilbert space. 

The trade-off is that the error in the de Finetti theorem of~\cite{BraHar-12} (see also~\cite{LiSmi-15}) is not quantified in the usual trace-class norm, and that the measure constructed there does not charge only bosonic states (i.e. the measure might live on mixed one-body operators $\gamma$, not just on pure states $|u\rangle \langle u |$). These are the two aspects we have to circumvent to conclude the proof along the lines of~\cite{LewNamRou-14c}.

As a final remark on the method, I stress that it is meant to obtain  \emph{variationally} the \emph{full statement} of Theorem~\ref{thm:main} for the \emph{fully general case} of~\eqref{eq:start hamil}. For $\beta < 1/d$ one might still obtain~\cite{Lewin-ICMP,LewNamRou-18a,LewNamRou-15} part of the statement under restrictive assumptions (typically one does not obtain the convergence of all density matrices and/or assumes that the limit problem has a unique minimizer and/or does not include the possibility of a magnetic field). See~\cite{Rougerie-EMS} for a review on these aspects. 

\section{Proof of the mean-field limit theorem}\label{sec:proof}

We follow the general strategy of~\cite{LewNamRou-14c} (also presented in~\cite[Chapter~7]{Rougerie-LMU,Rougerie-spartacus}), with the appropriate modifications allowing to insert the main new tool, Theorem~\ref{thm:deF inf} below.  

\subsection{Localizing the two-body Hamiltonian} Our first task is to localize the Hamiltonian to low one-body energy states. Let us denote 
\begin{equation}\label{eq:one body}
 h = \left( -\im \nabla + A \right) ^2 + V
\end{equation}
and, for some high-energy cut-off $\Lambda \in \R$, 
\begin{equation}\label{eq:spec proj}
P = \1_{h\leq \Lambda},\quad Q = \1 - P. 
\end{equation}
Recall that since $h$ has compact resolvent, $P L^2 (\R^d)$ is finite dimensional. In fact
\begin{equation}\label{eq:dim}
N_\Lambda := \dim (P) \leq C \Lambda ^{\frac{d}{s} + \frac{d}{2}}  
\end{equation}
with $s$ the exponent in~\eqref{eq:trap s}, see~\cite[Lemma~3.3]{LewNamRou-14c} and references therein. 

We shall write the many-body energy of a quasi-minimizer $\Psi_N$ in the manner 
\begin{equation}\label{eq:ener two}
\frac{1}{N} \langle \Psi_N |H_N| \Psi_N \rangle = \frac{1}{2}\tr\left( H_2 \gamma_N ^{(2)}\right) 
\end{equation}
using the two-particle reduced density matrix $\gamma_N^{(2)}$ and the two-body Hamiltonian 
$$ H_2 := h_1 + h_2 + w_{N,\beta} (x_1-x_2)$$
acting on $\gH^2 = L^2 (\R^d)^{\otimes 2}$, with $h_1 = h\otimes \1$ and $h_2 = \1 \otimes h$.

We shall need a slightly modified version: for $\eps >0$ let 
\begin{equation}\label{eq:eps two body}
H_2 ^\eps = H_2 - \eps N^{d\beta} \left|w \left(N^{\beta} \left(x-y\right)\right)\right| 
\end{equation}
Now we project the two-body Hamiltonian below the high-energy cut-off:

\begin{lemma}[\textbf{Localized two-body Hamiltonian}]\label{lem:GP localize-energy}\mbox{}\\
Assume that $\Lambda \ge C \eps^{-1} N^{d\beta}$ for a large enough constant $C>0$ and $0< \eps < 1$. Then we have, as operators on $L^2 (\R^{2d})$,   
\begin{align} \label{eq:GP H2-localized-error}
H_2 \geq  P^{\otimes 2} H_{2} ^\eps P^{\otimes 2} + \frac{\Lambda}{2} \left( Q \otimes \1 + \1 \otimes Q \right)
\end{align}
\end{lemma}

\begin{proof}
This is~\cite[Lemma~3.6]{LewNamRou-14c}. For brevity I do not reproduce the proof.
\end{proof}

\begin{remark}[Dependence on the trap]
The virtue of the quantum de Finetti theorem we use below is that only the logarithm of $N_\Lambda$ enters the estimate. For a polynomial trap as assumed in~\eqref{eq:trap s}, this gives a logarithmic dependence on $\Lambda$ as per~\eqref{eq:dim}. Such a dependence is quite negligible for it will come multiplied by negative powers of $N$, and $\Lambda$ depends polynomially on $N$ in the previous lemma. If the trap has a weaker growth (say logarithmic), one can still proceed by using adapted variants of~\eqref{eq:trap s}, but now $N_\Lambda$ depends exponentially on $\Lambda$ and the dependence of errors on $N_\Lambda$ becomes relevant. A bit of fine tuning is required to compute the dependence on the growth of $V$ of the maximal $\beta$ that the method allows. We do not pursue this for simplicity. \hfill$\diamond$ 
\end{remark}

\subsection{Quantum de Finetti} Now we apply the quantum de Finetti theorem whose proof is recalled in Appendix~\ref{sec:BraHar} below to the localized reduced $2$-body density matrix of a quasi-minimizer (or any other state for that matter):

\begin{proposition}[\textbf{de Finetti representation of projected density matrices}]\label{pro:deF loc}\mbox{}\\
Let $\gH$ be a complex separable Hilbert space, and $\gH_N = \gH ^{\otimes_{\sym} N}$ the corresponding bosonic space. Let 
$$\gamma_N ^{(2)} = \tr_{3\to N}  |\Psi_N\rangle \langle \Psi_N |$$
be the $2$-body reduced density matrix of a $N$-body state vector $\Psi_N \in \gH_N$ (or general mixed state). 

Let $P$ be a finite dimensional orthogonal projector. There exists a Borel measure $\mu_N ^{(2)}$ on the set of one-body mixed states
\begin{equation}\label{eq:one body states}
\cS_P := \left\{ \gamma \mbox{ positive trace-class operator on } P \gH, \, \tr \gamma = 1 \right\} 
\end{equation}
such that 
\begin{equation}\label{eq:deF rep}
\sup_{0\leq A,B \leq 1} \tr\left| A \otimes B \left(  P ^{\otimes 2} \gamma_N ^{(2)} P ^{\otimes 2} - \int \gamma ^{\otimes 2} d\mu_N ^{(2)} (\gamma) \right)\right| \leq C \sqrt{\frac{\log (\dim (P))}{N}}
\end{equation}
where the sup is over bounded operators on $P\gH$. 
\end{proposition}

Previous comparable statements in~\cite{LewNamRou-14c,LewNamRou-15} have an error proportional to $\dim(P)/N$. The point is that we shall be forced to apply the above with a rather large $\dim (P) \gg N ^{1/2}$, so that the much-improved dependence on $\dim (P)$ in~\eqref{eq:deF rep} counter-balances the worst dependence on~$N$. 

\begin{proof}
We combine Fock-space localization and the information-theoretic quantum de Finetti theorem recalled in Appendix~\ref{sec:BraHar}. 

\medskip

\noindent \textbf{Step 1, de Finetti.} Let $\Gamma_N$ be a mixed state over $P^{\otimes N} \gH_N$. From Theorem~\ref{thm:deF inf} we know there exists a probability measure $\mu_N ^{(2)}$ over $\cS_P$ such that
\begin{equation}\label{eq:deF proof}
\sup_{\Lambda_1,\Lambda_2 \in \cM (\gH)} \norm{\Lambda_1 \otimes \Lambda_2 \left(\Gamma_N ^{(2)}  - \int \gamma ^{\otimes 2} d\mu_N  ^{(2)}(\gamma) \right)}_{\gS ^1 (\gH^2)} \leq \sqrt{\frac{2 \log (\dim P)}{N-1}}
\end{equation}
where the sup is over quantum measurements  (see Definition~\ref{def:measurements}) and 
$$ \norm{A}_{\gS^1} = \tr |A|$$
is the trace-class norm~\cite{Simon-79,Schatten-60}. We claim that this implies 
\begin{equation}\label{eq:deF rep proof}
\sup_{0\leq A,B \leq 1} \tr\left| A \otimes B \left(   \Gamma_N ^{(2)}  - \int \gamma ^{\otimes 2} d\mu_N ^{(2)} (\gamma) \right)\right| \leq \sqrt{\frac{\log (\dim (P))}{N}}
\end{equation}
where the sup is now over bounded operators. Indeed, given operators $A_1,A_2$, define measurements 
$$
\Lambda_j (\gamma) = \tr \left[ A_j \gamma \right] |e_1 \rangle \langle e_1 | + \tr\left[ (\1 - A_j) \gamma\right] |e_2 \rangle \langle e_2 |
$$
for orthonormal vectors $e_1,e_2$ (independent of $j$). Then, for any $2$-particle operator $\gamma_{2}$, we have 
\begin{align*}
\Lambda_1 \otimes \Lambda_2 \, \gamma_{2} &= \tr\left[ A_1\otimes A_2 \gamma_{2} \right] |e_1 ^{\otimes 2} \rangle \langle e_1^{\otimes 2}| \\
&+ \sum_{(B_j,f_j)_{1\leq j \leq 2}} \tr\left[ B_{1} \otimes B_2 \gamma_{2} \right] |f_{1} \otimes f_2 \rangle \langle f_{1} \otimes f_2 | 
\end{align*}
where the last sum is over all possible choices of $B_j= A_j$ or $B_j = \1 - A_j$, with $f_j = e_1$ in the former case and $f_j = e_2$ in the second, and we impose that for at least one index $j$, $B_j = \1 - A_j$ (and thus $f_j = e_2$). Since $e_1$ and $e_2$ are orthogonal it follows that all projectors appearing in the second line live on spaces orthogonal to $e_1 ^{\otimes 2}$ and thus 
$$
\tr \left|\Lambda_1 \otimes \Lambda_2 \gamma_{2} \right| \geq \left|\tr\left[ A_1\otimes A_2 \gamma_{2} \right]\right|.
$$
Applying this to 
$$\gamma_{2} = \gamma_N ^{(2)} - \int \gamma ^{\otimes 2} d\mu_N ^{(2)} (\gamma)$$
shows that indeed~\eqref{eq:deF rep proof} follows from~\eqref{eq:deF proof}.

\medskip

\noindent \textbf{Step 2, localization.} From methods discussed e.g. in~\cite{Lewin-11} or~\cite[Chapter~5]{Rougerie-LMU,Rougerie-spartacus} we know there exists a state $\Gamma_N ^P$ on the truncated bosonic Fock space 
$$ 
\gF_P = \C \oplus P \gH \oplus P^{\otimes 2} \gH_2 \oplus \ldots \oplus P^{\otimes N} \gH_N 
$$
of the form 
$$ 
\Gamma_N^P = \left( c_{N,0} \Gamma_{N,0} ^P \right) \oplus \left( c_{N,1} \Gamma_{N,1} ^P \right) \oplus \ldots \oplus \left( c_{N,N} \Gamma_{N,N} ^P\right),
$$
with $c_{N,j} \geq 0$, $\sum_{j=0}^N c_{N,j} = 1$ and $\Gamma_{N,j}$ a $j$-particles state, such that 
\begin{equation}\label{eq:localization}
 P^{\otimes k} \gamma_N ^{(k)} P ^{\otimes k} = \sum_{\ell = k} ^N c_{N,\ell} {N \choose k} ^{-1} { \ell \choose k }  \left( \Gamma_{N,\ell} ^P \right) ^{(k)}. 
\end{equation}
We apply the previous step to each $\Gamma_{N,\ell}^P $, $\ell \geq 2$, obtaining probability measures $\mu_{N,\ell} ^{(2)}$ and the estimates   
\begin{equation}\label{eq:deF rep proof bis}
\sup_{0\leq A,B \leq 1} \tr\left| A \otimes B \left( \left( \Gamma_{N,\ell} ^P\right)  ^{(2)}  - \int \gamma ^{\otimes 2} d\mu_{N,\ell} ^{(2)} (\gamma) \right)\right| \leq C \sqrt{\frac{\log (\dim (P))}{\ell}}
\end{equation}
for $\ell \geq 2$. Setting 
$$ \mu_N ^{(2)} = \sum_{\ell = 2} ^N c_{N,\ell} {N \choose 2} ^{-1} { \ell \choose 2 } \mu_{N,\ell} ^{(2)} ,$$
combining~\eqref{eq:deF rep proof bis} with~\eqref{eq:localization} we get the statement, because 
$$ \frac{1}{\sqrt{\ell}}{N \choose k} ^{-1} { \ell \choose k } = \frac{1}{\sqrt{N}} \sqrt{\frac{\ell}{N}} \frac{(\ell-1) \ldots (\ell - k + 1 )  }{(N-1) \ldots (N - k + 1 )}\leq \frac{1}{\sqrt{N}}.$$
\end{proof}

\subsection{Mean-field functionals} The previous ingredients will allow to replace the $N$-body problem by a mean-field one, namely reduce attention to two-body matrices of the form $\gamma ^{\otimes 2}$ in~\eqref{eq:ener two}. This leads us to considering mixed Hartree and NLS functionals. First, let
$$ a = \int_{\R^d} w$$
and\footnote{The definition of $\gamma (x;x)$ is recalled in~\eqref{eq:gamma} below. Formally it is the integral kernel $\gamma(x;y)$ of $\gamma$ evaluated at $x=y$.}
\begin{equation}\label{eq:nls func mixed}
\cEnlsm [\gamma] := \tr \left( \left (-\im \nabla + A)^2 + V \right) \gamma \right) +  \frac{a}{2} \int_{\R^d} \left| \gamma (x;x) \right| ^2 dx 
\end{equation}
and
\begin{equation}\label{eq:nls ener mixed}
\Enlsm = \min \left\{ \cEnls [\gamma], \: \gamma \mbox{ trace-class operator on } L^2 (\R^d), \, \tr \gamma = 1 \right\}
\end{equation}
with $\Mnlsm$ the set of associated minimizers. The superscript $\mathrm{m}$ means mixed because we allow mixed states $\gamma$ as arguments. We recover the objects described in Theorem~\ref{thm:main} by reducing to pure states $\gamma = |u\rangle \langle u |$. Note that the minimization problems amongst all mixed states and amongst only pure states can differ~\cite{Seiringer-02,Seiringer-03}, especially if $A\neq 0$. This is a point we shall deal with later. 

For the moment we need to ensure that the Hartree problem, with smeared non-linearity, converges to the above. Let thus 
\begin{equation}\label{eq:H func}
\cEH [\gamma] := \tr \left( \left (-\im \nabla + A)^2 + V \right) \gamma \right) +  \frac{1}{2} \tr\left( w_{N,\beta} \, \gamma ^{\otimes 2}\right) 
\end{equation}
where (recall~\eqref{eq:w N beta}) $w_{N,\beta}$ is understood as a multiplication operator on $L^2 (\R^{2d})$. Also, let
\begin{equation}\label{eq:H ener}
\EH = \min \left\{ \cEH [\gamma], \: \gamma \mbox{ trace-class operator on } L^2 (\R^d), \, \tr \gamma = 1 \right\}.
\end{equation}
These are the objects one obtains by inserting a factorized ansatz in~\eqref{eq:ener two}. We shall need the following.

\begin{lemma}[\textbf{Stability of one-body functionals}]\label{lem:GP Hartree NLS}\mbox{}\\
Under the assumptions of Theorem~\ref{thm:main}, there is a constant $C>0$ such that, for any mixed one-body state $\gamma$
\begin{equation}
 \label{eq:GP hartree NLS 0}
\tr\left( h \gamma \right) \le C (\cEH [\gamma]+C) 
\end{equation}
and
\begin{equation}\label{eq:GP hartree NLS 1}
\left| \cEH [\gamma] - \cEnlsm[\gamma] \right| \le C N^{-\beta} \left( 1 + \tr\left( h \gamma \right) \right) ^2. 
\end{equation}
\end{lemma}

\begin{proof}
This is essentially similar to~\cite[Lemma~4.1]{LewNamRou-14c}. To extend the proof of~\eqref{eq:GP hartree NLS 1} to mixed states it is convenient to write the kernel of $\gamma$ as 
\begin{equation}\label{eq:gamma}
\gamma (x;y) = \sum_j \lambda_j u_j (x) \overline{u_j (y)}. 
\end{equation}
Then 
$$\tr\left( w_{N,\beta} \, \gamma ^{\otimes 2}\right) = \sum_{i,j} \lambda_i \lambda_j N ^{d\beta} \iint_{\R^d \times \R ^d} |u_i (x)|^2  w(N^{\beta} (x-y)) |u_j (y)| ^2 dx dy$$
and one may apply the arguments of the proof of~\cite[Lemma~4.1]{LewNamRou-14c} to each term of the sum. This is the place where we use that $|x| w (x) \in L^1 (\R ^d)$. 
\end{proof}

\subsection{Passage to the limit and conclusion} 
We now proceed to the 

\begin{proof}[Proof of Theorem~\ref{thm:main}]
The usual trial state argument (testing the energy with a factorized $\Psi_N = u ^{\otimes N}$) and Lemma~\ref{lem:GP Hartree NLS} give the energy upper bound 
\begin{equation}\label{eq:up bound}
E (N) \leq N \Enls + o (N). 
\end{equation}
We focus on the energy lower bound and associated convergence of reduced density matrices. Let $\Psi_N$ be a sequence of quasi-minimizers as in the statement of the theorem and $P$ the projector on low kinetic energy modes defined above.

\medskip 

\noindent\textbf{First energy estimate.} Use the Fourier transform to write a smooth pair potential $W$ as 
\begin{align}\label{eq:Fourier} 
W(x-y) &= \int_{\R^d} \widehat{W} (p) e ^{ip\cdot x} e^{-ip \cdot y} dp \nonumber \\
&= \int_{\R^d} \widehat{W} (p) \left( \sum_{(i,j) \in \{+,-\}} c_p ^i \otimes c_p^j + s_p ^i \otimes s_p ^j \right) dp
\end{align}
with $c_p ^{\pm},s_p ^{\pm}$ the bounded operators (with bound $1$) of multiplication by the positive and negative parts of $\cos ( p \cdot x), \sin (p \cdot x)$. 

Proposition~\ref{pro:deF loc} gives 
\begin{equation}\label{eq:deF bis}
\sup_{ 0 \leq A,B \leq \1} \left| \tr A\otimes B \left(P ^{\otimes 2}\gamma_N ^{(2)} P ^{\otimes 2} - \int \gamma ^{\otimes 2} d\mu_N ^{(2)}(\gamma) \right)\right| \leq C \sqrt{\frac{\log (\dim (P))}{N}}. 
\end{equation}   
But the sup in the above is bounded by a constant times the sup over signed operators $A,B$ so we may combine with~\eqref{eq:Fourier} and use the triangle inequality to obtain 
$$
\left|\tr \left[ W (x-y) \left(P ^{\otimes 2}\gamma_N ^{(2)} P ^{\otimes 2} - \int \gamma ^{\otimes 2} d\mu_N ^{(2)}(\gamma) \right)\right] \right| \leq C \sqrt{\frac{\log (\dim(P))}{N}} \int_{\R^d} |\widehat{W} (p)| dp. 
$$
Applying this with $W (x) = N ^{d\beta} w (N^{\beta} x)$ and $P$ as in~\eqref{eq:spec proj} we get 
$$ \tr \left( P^{\otimes 2} H_2^\eps P ^{\otimes 2} \gamma_N ^{(2)} \right) \geq  \int_{\cS_P} \tr \left(  H_2^\eps \gamma ^{\otimes 2} \right) d\mu_N ^{(2)} (\gamma) - C (N^{d\beta} + \Lambda ) \sqrt{\frac{\log (N_\Lambda)}{N}} $$
where we also used that on $P L ^2 (\R ^d)$, $h \leq \Lambda$ by definition to apply~\eqref{eq:deF bis} to the one-body term. Combining with Lemma~\ref{lem:GP localize-energy} yields our first energy lower bound
\begin{equation}\label{eq:first lower bound}
\frac{E(N)}{N} + o (1) \geq \frac{1}{2}\tr \left( H_2 \gamma_N ^{(2)} \right) \geq  \int_{\cS_P} \cEH_\eps [\gamma] d\mu_N ^{(2)} (\gamma) + C \Lambda \tr \left( Q \gamma_N ^{(1)} \right) - C (N^{d\beta} + \Lambda ) \sqrt{\frac{\log (N_\Lambda)}{N}} 
\end{equation}
with
$$
\cEH_\eps [\gamma] := \tr \left( \left (-\im \nabla + A)^2 + V \right) \gamma \right) +  \frac{1}{2} \tr\left( w_{N,\beta} \, \gamma ^{\otimes 2}\right) -  \frac{\eps}{2} \tr\left( |w_{N,\beta}| \gamma ^{\otimes 2}\right).
$$
It is easy to see that the latter functional is bounded below uniformly in $N$  if we choose  $\eps$ small enough (but independent of $N$), which we henceforth do. We may now set 
\begin{equation}\label{eq:Lambda}
 \Lambda = C N ^{d\beta} \eps ^{-1} 
\end{equation}
for a large constant $C$. Then, inserting~\eqref{eq:dim} and the uniform lower bound on $\cEH_\eps [\gamma]$ in~\eqref{eq:first lower bound} we obtain
\begin{equation}\label{eq:tight 1}
 C \geq \Lambda \tr \left( Q \gamma_N ^{(1)} \right) - C_\eps o_N (1) 
\end{equation}
under the assumption that 
$$
\beta < \frac{1}{2d}.
$$ 
The constant $C_\eps$ depends only on $\eps$. We deduce that 
\begin{equation}\label{eq:tightness}
 \tr \left( Q \gamma_N ^{(1)} \right) \to 0 
\end{equation}
when $N\to \infty$.

\medskip 

\noindent\textbf{The de Finetti measures converge.} Returning to the proof of Proposition~\ref{pro:deF loc} we have that (recall that $\sum_\ell c_{N,\ell} = 1$)
\begin{align}\label{eq:tight 2}
 \int_{\cS_P} d\mu_N ^{(2)} (\gamma) &= \tr \left( P ^{\otimes 2} \gamma_N^{(2)} P ^{\otimes 2}\right) = \sum_{\ell= 2} ^N c_{N,\ell} \frac{\ell (\ell - 1)}{N(N-1)}\nonumber \\
 &\geq  \sum_{\ell= 0} ^N c_{N,\ell} \frac{\ell^2 }{N^2} - \frac{C}{N}\geq \left(\sum_{\ell= 0} ^N c_{N,\ell} \frac{\ell}{N}\right) ^2 - \frac{C}{N} \nonumber \\
 &= \left( \tr \left( P \gamma_N^{(1)} P \right) \right) ^2 - \frac{C}{N}
\end{align}
where we used Jensen's inequality. 
But~\eqref{eq:tightness} implies 
$$ \tr \left( P \gamma_N^{(1)} P \right) \underset{N\to\infty}{\longrightarrow} 1.$$
Returning to~\eqref{eq:tight 2} we have that 
$$
\int_{\cS_P} d\mu_N ^{(2)} (\gamma) \underset{N\to\infty}{\longrightarrow} 1.$$
Thus the sequence $(\mu_N^{(2)})_N$ of measures given by Proposition~\ref{pro:deF loc} is tight on the set of one-body mixed states
$$ 
\cS := \left\{ \gamma \mbox{ positive trace-class operator on } L^2 (\R^d), \, \tr \gamma = 1 \right\}.
$$
Modulo subsequence $(\mu_N^{(2)})_N$ converges to a measure $\mu$.

\medskip

\noindent\textbf{Convergence of reduced density matrices.} In this step we reproduce for convenience arguments already used repeatedly in~\cite{LewNamRou-14} and~\cite{Rougerie-LMU,Rougerie-spartacus}. We may return to~\eqref{eq:first lower bound} and derive a similar energy lower bound to 
\begin{equation}\label{eq:hamil eta}
 \tr \left( \left(H_2 - \eta h \otimes \1 - \eta \1 \otimes h \right) \gamma^{(2)}_N \right)  
\end{equation}
for some small fixed $\eta >0$. For $\eta$ small enough we may use a variant of Lemma~\ref{lem:GP Hartree NLS} to deduce that~\eqref{eq:hamil eta} is uniformly bounded from below. Hence, combining with~\eqref{eq:ener two} and the energy upper bound~\eqref{eq:up bound} we deduce that 
$$ \frac{\eta}{2} \tr \left( \left( h \otimes \1 + \1 \otimes h \right) \gamma_N ^{(2)} \right) = \eta \tr \left( h \gamma_N ^{(1)} \right) \leq C_{\eps,\eta}.
$$
Since $h$ has compact resolvent we deduce (modulo subsequence) that 
$$ \gamma_N ^{(1)} \to \gamma ^{(1)}$$
strongly in trace-class, for some limit one-body bosonic density matrix $\gamma ^{(1)}$. But we also have (again, modulo subsequences) 
$$ \gamma_N ^{(k)} \underset{\star}{\wto} \gamma ^{(k)}$$
weakly-$\star$ in the trace-class. Applying the weak quantum de Finetti theorem~\cite[Theorem~2.2]{LewNamRou-14} we deduce that there exists a measure $\nu$ on the unit ball of $L^2 (\R^d)$ such that 
$$
\gamma ^{(k)} = \int |u ^{\otimes k } \rangle \langle u ^{\otimes k}| d\nu (u).
$$
But since $\gamma ^{(1)}$ must have trace $1$, the measure $\nu$ must actually live on 
$$ S L^2 (\R ^d) = \left\{ u\in L^2 (\R^d),\: \int_{\R^d} |u| ^2 = 1 \right\},$$
the unit sphere of $L^2 (\R^d)$. 

Next we claim that the two measures $\mu$ and $\nu$ just found are related by 
\begin{equation}\label{eq:pouet}
 \int |u ^{\otimes 2 } \rangle \langle u ^{\otimes 2}| d\nu (u) = \int |u ^{\otimes 2 } \rangle \langle u ^{\otimes 2}| d\mu (u). 
\end{equation}
Indeed, let 
$$\tilde{P} = \1_{h\leq \tilde{\Lambda}}$$
where $\tilde{\Lambda}$ is a fixed cut-off (different from $\Lambda$ above). Testing~\eqref{eq:deF rep} with $A_1,A_2$ finite rank operators whose ranges lie within that of $\tilde{P}$ we get
$$
\tr \left( A_1 \otimes A_2 \gamma_N ^{(2)}\right) \underset{N\to\infty}{\longrightarrow} \tr \left( A_1 \otimes A_2 \int_{\cS} \gamma ^{\otimes 2} d\mu (\gamma) \right)
$$
using the convergence of $\mu_N ^{(2)}$ to $\mu$. On the other hand, by the convergence of $\gamma_N ^{(2)}$ to $\gamma^{(2)}$ we also have 
$$
\tr \left( A_1 \otimes A_2 \gamma_N ^{(2)}\right) \underset{N\to\infty}{\longrightarrow} \tr \left( A_1 \otimes A_2 \gamma ^{(2)} \right) =  \tr \left( A_1 \otimes A_2 \int_{S L^2 (\R^d)} |u ^{\otimes 2} \rangle \langle u^{\otimes 2}| d\nu (u)\right). 
$$
Thus 
\begin{equation}\label{eq:id meas}
 \tr \left( A_1 \otimes A_2 \int_{\cS} \gamma ^{\otimes 2} d\mu (\gamma) \right) = \tr \left( A_1 \otimes A_2 \int_{S L^2 (\R^d)} |u ^{\otimes 2} \rangle \langle u^{\otimes 2}| d\nu (u)\right) 
\end{equation}
for any $A_1,A_2$ with range within that of $\tilde{P}$. Letting finally $\tilde{\Lambda} \to \infty$ yields $\tilde{P} \to \1$ and thus~\eqref{eq:id meas} holds for any compact operators $A_1,A_2$. This implies~\eqref{eq:pouet}. In particular, since the left-hand side of~\eqref{eq:pouet} is $\gamma ^{(2)}$, a bosonic operator,  $\mu$ must be supported on pure states $\gamma = |u\rangle \langle u|$, see~\cite{HudMoo-75}.   

\medskip 

\noindent\textbf{Final passage to the liminf}. Let us return to~\eqref{eq:first lower bound}. We split the integral over one-body states $\gamma$ between low and high kinetic energy states:
$$ \texttt{Low} = \left\{ \gamma \in \cS, \, \tr \left( h \gamma \right) \leq C_{\texttt{Kin}} \right\}, \quad \texttt{High} = \cS \setminus \texttt{Low} 
$$
with $C_{\texttt{Kin}}$ a constant independent of $N$ (we will take $C_{\texttt{Kin}}\to \infty$ after $N\to \infty$). Using Lemma~\ref{lem:GP Hartree NLS} (or rather an obvious variant applying to $\cEH_{\eps}$) we obtain 
\begin{align*}
\int_{\cS_P} \cEH_\eps [\gamma] d\mu_N ^{(2)} (\gamma) &\geq C_\eps  C_{\texttt{Kin}} \int_{\texttt{High}} N ^{\beta /4} d\mu_N ^{(2)} (\gamma) + \int_{\texttt{Low}} \cEnlsm_{\eps} [\gamma] d\mu_N ^{(2)} (\gamma) - C_\eps N ^{-\beta}\\
&\geq \int_{\cS_P} \min\left( C_\eps C_{\texttt{Kin}}, \cEnlsm_{\eps} [\gamma] \right)  d\mu_N ^{(2)} (\gamma) - C_\eps N ^{-\beta}
\end{align*}
where $\cEnlsm_{\eps}$ is $\cEnlsm$ with 
$$ a \rightsquigarrow a - \eps a.$$
Inserting in~\eqref{eq:first lower bound} and passing to the liminf in $N\to \infty$ this implies 
$$ \Enls \geq \underset{N\to \infty}{\liminf} \frac{E(N)}{N} \geq \int_{\cS} \min\left( C_\eps C_{\texttt{Kin}}, \cEnls_{m,\eps}  [\gamma] \right) d\mu (\gamma).$$
Finally, we pass to the limit $C_{\texttt{Kin}} \to \infty$ and then the limit $\eps \to 0$ to deduce 
\begin{equation}\label{eq:final low}
 \Enls \geq \underset{N\to \infty}{\liminf} \frac{E(N)}{N} \geq \int_{\cS} \cEnls_{m} [\gamma] d\mu (\gamma). 
\end{equation}
But as we saw above $\mu$ must be supported on pure states $\gamma = |u\rangle \langle u |$, which yields both the energy lower bound concluding the proof of~\eqref{eq:resul ener} and the fact that $\mu$ must be supported on $\Mnls$. Because $\cEnls_{m} [\gamma]$ is a linear function of $\gamma ^{\otimes 2}$ we can also combine~\eqref{eq:final low} with~\eqref{eq:pouet} to deduce that also $\nu$ must be supported on $\Mnls$, which proves~\eqref{eq:result DM}.
\end{proof}

\newpage

\appendix

\section{An information-theoretic quantum de Finetti theorem}\label{sec:BraHar}

Here we reproduce, for the convenience of the reader, the statement and proof of a Theorem of Brand\~{a}o and Harrow. No claim of originality is thus made. See also the lecture notes~\cite{BraChrHarWal-16} and~\cite{LiSmi-15}, which contains results related to~\cite{BraHar-12}.

\subsection{A local de Finetti theorem} In~\cite{BraHar-12}, Brand\~{a}o and Harrow proved a quantitative quantum de Finetti theorem, where the quality of the approximation deteriorates only logarithmically with the dimension of the one-body state space, in contrast with previous results~\cite{Chiribella-11,ChrKonMitRen-07,Harrow-13,KonRen-05,LewNamRou-14b}. The trade-off is that the control is in a weaker norm than trace-class. 

\medskip

Recall that, for a complex separable Hilbert space $\gH$, the state space is 
\begin{equation}\label{eq:state space}
\cS (\gH) := \left\{ \gamma \mbox{ positive trace-class operator on } \gH, \, \tr \gamma = 1 \right\}. 
\end{equation}
We shall need a notion of measurement of such states:

\begin{definition}[\textbf{Quantum measurements}]\mbox{}\label{def:measurements}\\
A quantum measurement $\Lambda$ on a complex Hilbert space $\gH$ of dimension $d$ is identified with a set $(M_k,e_k)_{k=1\ldots d}$ of bounded operators and vectors such that 
\begin{itemize}
\item $M_k \geq 0$ for all $k$ and $\sum_k M_k = \1$ 
\item $(e_k)$ is an orthonormal basis of $\gH$.
\end{itemize}
Its action on a state $\rho \in \cS (\gH)$ is given by 
\begin{equation}\label{eq:measurement}
\Lambda (\rho) = \sum_k \tr [M_k \rho] | e_k \rangle \langle e_k |. 
\end{equation}
We shall denote $\cM (\gH)$ the set of quantum measurements on $\gH$. 
\end{definition}

The definition is usually generalized to a map from states of a Hilbert space $\gH_1$ to states of another Hilbert space $\gH_2$ by taking $(e_k)$ to be an orthonormal basis of the latter.

Given two measurements $\Lambda_1,\Lambda_2$ on $\gH_1,\gH_2$ one can define $\Lambda_1 \otimes \Lambda_2$ on $\gH_1 \otimes \gH_2$ in the natural way: it is associated with the operators $M_{k,1} \otimes M_{k,2}$ and vectors $e_{k,1}\otimes e_{k,2}$.  One can also define $\1 \otimes \Lambda_2$ as a measurement on $\gH_1 \otimes \gH_2$ by setting
\begin{equation}\label{eq:one times measure} 
\1 \otimes \Lambda_2 \Gamma = \sum_{k} \tr_{2} [ M_{k,2} \Gamma ] \otimes | e_{k,2} \rangle \langle e_{k,2}|. 
\end{equation}

\medskip

The statement we wish to discuss applies to more general states than the bosonic ones encountered in the context of mean-field limits:

\begin{definition}[\textbf{Symmetric $N$-body states}]\label{def:sym state}\mbox{}\\
Let $\gH$ be a complex separable Hilbert space. A symmetric $N$-particles state is a state $\Gamma$ over $\gH ^{\otimes N}$ commuting with label permutations: 
$$ \Gamma \in \cS (\gH ^{\otimes N}) \mbox{ such that } U_{\sigma} \Gamma = \Gamma U_{\sigma}$$
for all permutation $\sigma$, where $U_\sigma$ is the unitary operator exchanging labels according to
$$ U_{\sigma} u_1\otimes \ldots \otimes u_N = u_{\sigma(1)} \otimes \ldots \otimes u_{\sigma (N)}.$$
\end{definition}

A bosonic state satisfies the stronger condition
$$ U_{\sigma} \Gamma = \Gamma U_{\sigma} = \Gamma$$
for all permutation $\sigma$, and can thus be restricted to act only on the symmetric subspace $\gH ^{\otimes_{\rm sym} N}$, which is the point of view adopted in the main text. The reduced density matrices of a symmetric $N$-body state are defined as usual:
$$ \tr \left( B_k \Gamma ^{(k)} \right) := \tr \left( B_k \otimes \1 ^{\otimes (N-k)} \Gamma \right)$$
for any bounded operator $B_k$ on $\gH ^{\otimes k}$. Thus $\Gamma^{(k)}$ is a symmetric state on $\gH^{\otimes k}$.

\medskip

The rest of the appendix is concerned with exposing the proof (due to~\cite{BraHar-12,LiSmi-15}) of the following statement: 

\begin{theorem}[\textbf{Quantum de Finetti under local measurements}]\mbox{}\label{thm:BranHar}\\
Let $\gH$ be a finite dimensional complex Hilbert space, with dimension $d$. Let $\Gamma$ be a symmetric $N$-particles state on $\gH ^{\otimes N}$. For every $0\leq k \leq N$ there exists a probability measure $\mu_k$ on one-particle states such that 
\begin{equation}\label{eq:deF}
\sup_{\Lambda_1,\ldots,\Lambda_k\in \cM (\gH)} \norm{\Lambda_1 \otimes \ldots \otimes \Lambda_k \left(\Gamma ^{(k)}  - \int \gamma ^{\otimes k} d\mu_k (\gamma) \right)}_{\gS ^1 (\gH^k)} \leq \sqrt{\frac{2(k-1) ^2 \log d}{N-k+1}}. 
\end{equation}
The norm in the above is the trace-class norm $\norm{A}_{\gS^1} := \tr |A|.$
\end{theorem}

We start by explaining, in Section~\ref{sec:const}, how the measure is constructed. Then we state a more ``information-theoretic version'' of Theorem~\ref{thm:BranHar}, and proceed to its proof in Section~\ref{sec:BraHar proof}. Standard tools from quantum information theory are interjected in Section~\ref{sec:tools}, that the familiarized may want to skip on first reading.

\subsection{Construction}\label{sec:const} Denote $\cS (\gH) $ the space of states over a Hilbert space $\gH$, and $\cP (\cS (\gH))$ the set of probability measures on it. Theorem~\ref{thm:BranHar} is implied by 
\begin{equation}\label{eq:deF ref}
\inf_{\mu \in \cP (\cS (\gH))} \sup_{\Lambda_1,\ldots,\Lambda_k\in \cM (\gH)} \norm{\Lambda_1 \otimes \ldots \otimes \Lambda_k \left(\Gamma ^{(k)}  - \int \gamma ^{\otimes k} d\mu(\gamma) \right)}^2_{\gS ^1 (\gH^k)} \leq \frac{2k ^2 \log d}{N-k}. 
\end{equation}
Let $\cE$ be a quantum measurement over $\gH ^{\otimes (N-k)}$ acting as 
$$
\cE (\rho) = \sum_\mu \tr [M_\mu \rho] |e_\mu \rangle \langle e_\mu|.
$$
Defining, for each $\mu$,
$$ \Gamma_{\mu} := \frac{\tr_{k+1\to N} \left[ M_\mu \Gamma \right] \otimes |e_\mu \rangle \langle e_\mu |}{p_\mu},  \quad p_\mu := \tr \left[ \1^{\otimes k} \otimes M_\mu \Gamma \right]$$
we have by simple computations the

\begin{lemma}[\textbf{Decomposition of the state $\Gamma$}]\label{lem:decomp mu}\mbox{}\\
Let $\Gamma$ be a $N$-body state and $\cE$ a measurement on $\gH ^{\otimes (N-k)}$. For each $\mu$, $\Gamma_\mu$ defined as above is a $N$-body state. Moreover, 
$$ \sum_\mu p_\mu \Gamma_\mu = \1^{\otimes k} \otimes \cE \:\Gamma$$
and 
$$ \sum_\mu p_\mu \Gamma_\mu ^{(k)} = \Gamma ^{(k)}.$$
In particular $p_\mu \geq 0$ and  
$$ \sum_\mu p_\mu = 1.$$
\end{lemma}
%

Notice that $\Gamma$ being symmetric implies that 
$$
\Gamma_\mu ^{(k)}:= \frac{\tr_{k+1\to N} \left[ M_\mu \Gamma \right]}{p_\mu}
$$ also is. Now, the $N$-body state 
\begin{equation}\label{eq:Gamma tilde}
\Gammat_\cE:= \sum_\mu p_\mu \left(\Gamma_\mu ^{(1)}\right) ^{\otimes N}
\end{equation}
is certainly of the de Finetti form
$$ \Gammat_\cE = \int \gamma ^{\otimes N} d P (\gamma)$$
with $P$ a probability measure on one-body states. Thus   
\begin{multline}\label{eq:deF ref bis}
\inf_{\mu \in \cP (\cS (\gH))} \sup_{\Lambda_1,\ldots,\Lambda_k\in \cM (\gH)} \norm{\Lambda_1 \otimes \ldots \otimes \Lambda_k \left(\Gamma ^{(k)}  - \int \gamma ^{\otimes k} d\mu(\gamma) \right)}^2_{\gS ^1 (\gH^k)}\\
\leq \inf_{\cE \in \cM (\gH^{\otimes N-k})} \sup_{\Lambda_1,\ldots,\Lambda_k\in \cM (\gH)} \norm{\Lambda_1 \otimes \ldots \otimes \Lambda_k \left(\Gamma ^{(k)}  - \Gammat_\cE ^{(k)}\right)}^2_{\gS ^1 (\gH^k)}. 
\end{multline}

This is the first main idea: the measure is constructed by minimizing over measurements as above. The second main idea is to make a detour from the trace-class norm to more information-based measures, such as quantum relative entropies. In fact Theorem~\ref{thm:BranHar} is implied by an estimate of the error using the von Neumann relative entropy
\begin{equation}\label{eq:von Neumann}
 \cH (\Gamma,\Gamma') := \tr\left[ \Gamma \left( \log \Gamma - \log \Gamma' \right)\right]. 
\end{equation}

%

\begin{theorem}[\textbf{Quantum de Finetti, information-theoretic version}]\label{thm:deF inf}\mbox{}\\
Denoting 
$$\bL_k = \Lambda_1 \otimes \ldots \otimes \Lambda_k$$
we have  
\begin{equation}\label{eq:deF inf}
\inf_{\cE \in \cM (\gH^{\otimes (N-k)})} \sup_{\Lambda_1,\ldots,\Lambda_k\in \cM (\gH)} \sum_\mu p_\mu \cH \left( \bL_k \Gamma_{\mu} ^{(k)}, \bL_k \left( \Gamma_{\mu} ^{(1)} \right) ^{\otimes k}\right) \leq \frac{(k-1) ^2 \log d }{N-k +1}
\end{equation}
\end{theorem}

In the next subsection we provide more background and tools bearing on the multipartite mutual information of a state (i.e. on the left-hand side of~\eqref{eq:deF inf}). Then we give the proof of Theorem~\ref{thm:deF inf} and conclude that of Theorem~\ref{thm:BranHar} in Section~\ref{sec:BraHar proof}.

\subsection{Quantum information-theoretic tools}\label{sec:tools} In the sequel, 
$$
S(\Gamma) = -\tr\left[ \Gamma \log \Gamma\right]
$$ 
and $\cH$ as in~\eqref{eq:von Neumann} stand for the usual von Neumann entropy and relative entropy. A state over a $k$-fold tensor product will be denoted $\Gamma ^{1\ldots k}$ and for $l\leq k$, the reduced states $\Gamma ^{1\ldots l}$ over $l$-fold tensor products are defined as reduced density matrices, taking partial traces.

We start with the simple

\begin{lemma}[\textbf{Partial measurements}]\label{lem:part meas}\mbox{}\\
For a three partite state $\Gamma ^{123}$ and a measurement over the third system $\Lambda_3$
$$\left( \1 \otimes \1 \otimes \Lambda_3 \Gamma ^{123}\right) ^{12} = \Gamma ^{12}.$$
\end{lemma}

\begin{proof}
This is immediately seen from the definition~\eqref{eq:one times measure}
\end{proof}

\begin{definition}[\textbf{Mutual informations}]\label{def:mutual info}\mbox{}\\
For a quantum state $\Gamma = \Gamma ^{1\ldots k}$ over a tensor product $\gH_1 \otimes \ldots \otimes \gH_k$, define the \emph{bi-partite mutual informations} ($l\leq k$)
\begin{align*}
I (\gH_1, \ldots, \gH_l : \gH_{l+1}, \ldots, \gH_k)_\Gamma &= \cH \left(\Gamma ^{1\ldots k}, \Gamma ^{1\ldots l} \otimes \Gamma ^{l+1 \ldots k}  \right)\\
&= S\left( \Gamma ^{1\ldots l} \right) + S\left(\Gamma ^{l +1 \ldots k}\right) - S\left(\Gamma ^{1 \ldots k}\right)  
\end{align*}
and the \emph{multipartite mutual information}   
\begin{align*}
I (\gH_1 : \ldots : \gH_k)_\Gamma &= \cH \left(\Gamma ^{1\ldots k}, \Gamma ^{1} \otimes \ldots \otimes \Gamma ^{k}  \right)\\
&= S\left( \Gamma ^{1} \right) +\ldots + S\left(\Gamma ^{k}\right) - S\left(\Gamma ^{1 \ldots k}\right).  
\end{align*}
\end{definition}

\begin{remark}[Mutual informations]\label{rem:mutual info}\mbox{}\\
One can of course define mutual informations over any kind of partition. It follows from their definitions as relative entropies that these are positive quantities. The second equality in each definition comes from the fact that 
$$ \log (\Gamma ^1 \otimes \Gamma ^2) = \log \left( \Gamma ^1 \otimes \1 \right)+ \log \left( \1 \otimes \Gamma ^2\right)$$
that one uses to prove that the von Neumann entropy is subadditive (which is the same as the bipartite mutual information being positive).\hfill$\diamond$
\end{remark}

Mutual informations are positive, as we just saw, but they cannot be too big:

\begin{lemma}[\textbf{Bound on bi-partite mutual information}]\label{lem:bound info}\mbox{}\\
Let $\Gamma ^{12}$ be a bi-partite state over $\gH_1 \otimes \gH_2$, for two finite-dimensional Hilbert spaces of dimensions $d_1$ and $d_2$ respectively. Then 
$$ I (\gH_1 : \gH_2)_{\Gamma ^{12}} \leq 2 \min \left(\log d_1, \log d_2 \right).$$
\end{lemma}

\begin{proof}
Recall the Araki-Lieb inequality (proved by purification, Schmidt decomposition of pure states and subaddivity of  entropy~\cite{Lieb-75}) 
$$ S (\Gamma ^{12}) \geq \left| S (\Gamma ^1) - S (\Gamma ^2) \right|.$$
Inserting this in the (second) definition given above we have 
$$ I (\gH_1 : \gH_2)_{\Gamma ^{12}} =  S (\Gamma ^1) + S (\Gamma ^2)  - S (\Gamma ^{12}) \leq 2 \min \left( S (\Gamma ^1), S (\Gamma ^2) \right)$$
and the result follows, for the maximal entropy of a state in dimension $d$ is $\log d$.
\end{proof}

We also have the useful 

\begin{lemma}[\textbf{Mutual informations, bipartite to multipartite}]\label{lem:bi to multi}\mbox{}\\
 With the notation of the above definition
 \begin{equation}\label{eq:bi to multi}
 I (\gH_1 : \ldots :  \gH_k)_{\Gamma ^{1\ldots k}} = \sum_{j=2} ^k I(\gH_1, \ldots, \gH_{j-1} : \gH_j)_{\Gamma ^{1\ldots j}}.  
 \end{equation}
\end{lemma}

\begin{proof}
The definitions easily yield
$$ I (\gH_1 : \ldots :  \gH_k)_{\Gamma ^{1\ldots k}} = I (\gH_1, \ldots \gH_{k-1} :  \gH_k)_{\Gamma ^{1\ldots k}}  + I (\gH_1 : \ldots : \gH_{k-1} )_{\Gamma ^{1\ldots k-1}} $$
and it suffices to iterate this relation.
\end{proof}

Next we need the 

\begin{definition}[\textbf{Conditional mutual information}]\label{def:cond mut inf}\mbox{}\\
For a quantum state $\Gamma = \Gamma ^{123}$ over a three-fold tensor product $\gH_1 \otimes \gH_2 \otimes \gH_3$, define the \emph{conditional mutual information} 
\begin{align*}
I (\gH_1 : \gH_{2} | \gH_3)_\Gamma &= I (\gH_1 : \gH_2, \gH_3)_{\Gamma^{123}} - I (\gH_1 : \gH_3)_{\Gamma ^{13}}\\
&= \cH (\Gamma ^{123}, \Gamma^1 \otimes \Gamma ^{23}) - \cH (\Gamma^{13},\Gamma ^1 \otimes \Gamma ^3).
\end{align*}
\end{definition}

\begin{remark}[Conditional mutual information]\label{rem:cond mutual info}\mbox{}\\
It turns out that also the conditional mutual information is positive. This is not quite trivial and in fact follows from strong subaddivity of quantum entropy~\cite{LieRus-73a,LieRus-73b}.\hfill$\diamond$
\end{remark}

A first lemma bearing on the conditional mutual information is 

\begin{lemma}[\textbf{Conditional mutual information of partly measured states}]\label{lem:cond mut class}\mbox{}\\
Let $\Gamma ^{123}$ be a tri-partite state of the form
$$ \Gamma ^{123} = \sum_j p_j \Gamma^{12}_j \otimes |e_j \rangle \langle e_j|$$
for positive numbers $p_j$ summing to $1$, bi-partite states $\Gamma^{12}_j$ and an orthonormal basis $(e_j)$. Then 
\begin{equation}\label{eq:cond mut class}
I (\gH_1 : \gH_{2} | \gH_3)_\Gamma = \sum_j p_j I(\gH^1 : \gH ^2)_{\Gamma_j^{12}}. 
\end{equation}
\end{lemma}

\begin{proof}
Diagonalize $\Gamma^{12}_j$ and then Schmidt-decompose its eigenvectors. This gives 
$$
\Gamma ^{123} = \sum_{j,k,l,m} p_j q_{j,k} \sqrt{r_{j,k,l}} \sqrt{r_{j,k,m}} |a_{j,k,l} \rangle \langle a_{j,k,m} | \otimes |b_{j,k,l} \rangle \langle b_{j,k,m} | \otimes |e_j \rangle \langle e_j|. 
$$
with orthonormal basis $(a_{j,k,l})_l$ and $(b_{j,k,l})_l$ and positive numbers satisfying
$$  \sum_k q_{j,k} = 1, \quad \sum_{l} r_{j,k,l} = 1.$$
Inserting this in the definitions 
\begin{align*}
I (\gH_1 : \gH_{2} | \gH_3)_\Gamma &= - S (\Gamma ^{123}) + S (\Gamma^{23}) - S (\Gamma^3) + S (\Gamma ^{13}) \\
\sum_j p_j I(\gH^1 : \gH ^2)_{\Gamma^{12}_j}&= \sum_j p_j \left( S (\Gamma^1_j) + S (\Gamma^2_j) - S (\Gamma^{12}_j)\right), 
\end{align*}
the proof is a straightforward calculation.
\end{proof}

The previous lemma morally justifies an extension of Definition~\ref{def:cond mut inf}:

\begin{definition}[\textbf{Multi-partite conditional mutual information}]\label{def:cond mut inf bis}\mbox{}\\
Let $\Gamma ^{1\ldots k+1}$ be a $k+1$-particle state of the form 
$$ \Gamma ^{1\ldots k+1} = \sum_j p_j \Gamma_j ^{1 \ldots k} \otimes |e_j \rangle \langle e_j|$$
for positive numbers $p_j$ summing to $1$, $k$-partite states $\Gamma^{1\ldots k}_j$ and an orthonormal basis $(e_j)$. By definition 
$$ I (\gH_1 : \ldots : \gH_k | \gH_{k+1})_{\Gamma ^{1\ldots k+1}} = \sum_j p_j I (\gH_1 : \ldots : \gH_k )_{\Gamma_j ^{1\ldots k}}.$$
\end{definition}

Next we state a crucial link between bipartite mutual informations and conditional mutual informations. 
\begin{lemma}[\textbf{Chain rule for mutual informations}]\label{lem:chain rule}\mbox{}\\
Let $\Gamma = \Gamma ^{1\ldots N}$ be a $N$-partite state. 
\begin{equation}\label{eq:chain rule}
I (\gH_1, \ldots, \gH_{k-1} : \gH_{k},\ldots,\gH_N)_\Gamma = \sum_{j=k} ^N I(\gH_1,\ldots,\gH_{k-1} : \gH_j | \gH_{j+1} \otimes \ldots \otimes \gH_N )_{\Gamma ^{1 \ldots k-1 j \ldots N}}.  
\end{equation}
\end{lemma}

Henceforth the notation is that
\begin{align*}
 I (\gH_1, \ldots, \gH_{k-1} : \gH_{k},\ldots,\gH_N)_\Gamma &= I (\gH_1 \otimes \ldots \otimes \gH_{k-1} : \gH_{k},\ldots,\gH_N)_\Gamma\\
I (\gH_1,\ldots,\gH_{k-1} : \gH_j | \gH_{j+1} \otimes \ldots \otimes \gH_N )_{\Gamma ^{1 \ldots k-1 j \ldots N}} &= I (\gH_1\otimes \ldots \otimes \gH_{k-1} : \gH_j | \gH_{j+1} \otimes \ldots \otimes \gH_N )_{\Gamma ^{1 \ldots k-1 j \ldots N}}
\end{align*}
and similar conventions.

\begin{proof}
Call $M = N-k +1 $, $A = \gH_1 \otimes \ldots \otimes \gH_{k-1}$, $B_j = \gH_{k+j-1}.$ A reformulation is then 
$$ I (A : B_1, \ldots, B_M) = \sum_{j=1} ^M I (A : B_j | B_{j+1} \otimes \ldots \otimes B_M).$$
The left-hand side is equal to 
$$ \mathrm{LHS} = S (\Gamma ^A) + S (\Gamma ^{B_1 \ldots B_M}) - S (\Gamma ^{AB_1\ldots B_M}).$$
But the $M-$th term of the left-hand side is just a mutual information with no conditioning
$$ \mathrm{RHS}_M =  S (\Gamma ^A) + S (\Gamma ^{B_M}) - S (\Gamma ^{AB_M}).$$
The other terms of the right-hand side are, for $k < M,$
\begin{align*}
 \mathrm{RHS}_k &= \cH (\Gamma ^{A B_k \ldots B_M}, \Gamma ^A \otimes \Gamma^{B_k \ldots B_M}) - \cH (\Gamma^{A B_{k+1}\ldots B_M}, \Gamma ^A \otimes \Gamma ^{B_{k+1} \ldots B_M}) \\
 &= S (\Gamma ^{B_k \ldots B_M}) - S (\Gamma ^{AB_k\ldots B_M}) - (S \Gamma ^{B_{k+1} \ldots B_M}) - S (\Gamma ^{AB_{k+1}\ldots B_M})
\end{align*}
and the result clearly follows.
\end{proof}

\subsection{Proof of the main estimate}\label{sec:BraHar proof}  This is the proof of~\cite{BraHar-12}, expanded so as to become more accessible.

\begin{proof}[Proof of Theorem~\ref{thm:deF inf}]
From now on we occasionally label the copies of the one-body Hilbert space, for we will sometimes deal with states that are not fully symmetric. We denote $\bL_k = \Lambda_1 \otimes \ldots \otimes \Lambda_k$ a measurement as in the statement. We also let $\cE$ be a measurement on $\gH^{\otimes(N-k)}$ and the associated $(\Gamma_\mu)$'s be defined as in Lemma~\ref{lem:decomp mu}.

\smallskip

\noindent\textbf{Step 1.} We split multipartite informations into bipartite ones using Lemma~\ref{lem:bi to multi}: 
$$I(\gH: \ldots: \gH)_{\bL_k \Gamma_{\mu} ^{(k)}} = \sum_{j=2} ^k I (\gH_1,\ldots,\gH_{j-1} : \gH_j)_{\bL_j \Gamma_{\mu} ^{(j)}}.$$
Then, by monotony of the relative entropy we know~\cite{OhyPet-93} that the mutual information decreases under local measurements, for they are trace-preserving completely positive\footnote{Positive suffices actually~\cite{MulRee-17}.} maps. Thus  
$$I(\gH: \ldots: \gH)_{\bL_k \Gamma_{\mu} ^{(k)}} \leq \sum_{j=2} ^k I (\gH_1,\ldots,\gH_{j-1} : \gH_j)_{\Lambda_j \Gamma_{\mu} ^{(j)}}$$
where we abuse notation by denoting 
$$ \Lambda_j \Gamma_{\mu} ^{(j)} = \1 \otimes \dots \otimes \Lambda_j \otimes \ldots \otimes \1 \: \Gamma_{\mu} ^{(j)}.$$
Next, using symmetry and monotony of the relative entropy under partial traces
$$
I(\gH: \ldots: \gH)_{\bL_k \Gamma_{\mu} ^{(k)}} \leq (k-1) I (\gH_1,\ldots,\gH_{k-1} : \gH_k)_{\Lambda_k \Gamma_{\mu} ^{(k)}}. 
$$
Multiplying by $p_\mu$, summing over $\mu$ and using Lemma~\ref{lem:cond mut class} this yields 
\begin{equation}\label{eq:step 1}
\sum_\mu p_\mu I(\gH: \ldots: \gH)_{\bL_k \Gamma_{\mu} ^{(k)}} \leq (k-1) I (\gH_1,\ldots,\gH_{k-1} : \gH_k | \gH_{k+1},\ldots , \gH_N)_{\Lambda_k \otimes \cE \, \Gamma}
\end{equation}

\medskip

\noindent \textbf{Step 2.} We claim that 
\begin{multline}\label{eq:step 2}
(N-k+1) \min_{\cE} \max_{\Lambda_k}  I (\gH_1,\ldots,\gH_{k-1} : \gH_k | \gH_{k+1},\ldots , \gH_N)_{\Lambda_k \otimes \cE \, \Gamma} \\ 
\leq \max_{\Lambda_k, \ldots, \Lambda_{N}} I(\gH_1,\ldots,\gH_{k-1} : \gH_{k}, \ldots, \gH_N)_{\nu}
\end{multline}
where 
$$\nu = \1^{\otimes k-1} \otimes \Lambda_k \otimes \Lambda_{k+1} \otimes \ldots \otimes \Lambda_N (\Gamma).$$
This is achieved by a particular choice of the measurements $\Lambda_{k+1},\ldots,\Lambda_N$ (in particular, the minimum over all measurements on $N-k$ systems is bounded above using tensorized measurements). We start from the right-hand side of~\eqref{eq:step 2} and use the chain rule, Lemma~\ref{lem:chain rule}:
\begin{equation}\label{eq:use chain}
 I(\gH_1,\ldots,\gH_{k-1} : \gH_{k}, \ldots, \gH_N)_{\nu} = \sum_{j=k} ^N I(\gH_1,\ldots,\gH_{k-1} : \gH_{j} | \gH_{j+1} \ldots, \gH_N)_{\nu_j} 
\end{equation}
where in the right-hand side
$$ \nu_j = \nu ^{1\ldots k-1 j \ldots N}.$$
Observe that, as per Lemma~\ref{lem:part meas}, the $j$-th term in the right-hand side of~\eqref{eq:use chain} does not depend on the measurements $\Lambda_k,\ldots,\Lambda_{j-1}$. We then choose $\Lambda_N$ to maximize the $N$-th term, $\Lambda_{N-1}$ to maximize the $N-1$-th term given $\Lambda_N$, etc ..., $\Lambda_{j}$ to maximize the $j$-th term given the previous choices of $\Lambda_N,\ldots,\Lambda_{j+1}$, and continue this way iteratively. Then we certainly have, for each term,  
$$ I(\gH_1,\ldots,\gH_{k-1} : \gH_{j} | \gH_{j+1} \ldots, \gH_N)_{\nu_j} \geq \min_{\Lambda_{j+1},\ldots,\Lambda_N} \max_{\Lambda_j} I(\gH_1,\ldots,\gH_{k-1} : \gH_{j} | \gH_{j+1} \ldots, \gH_N)_{\nu_j} $$
and using Lemma~\ref{lem:part meas} again we also have 
\begin{multline*}
 I(\gH_1,\ldots,\gH_{k-1} : \gH_{j} | \gH_{j+1} \ldots, \gH_N)_{\nu_j} \\ \geq \min_{\Lambda_k,\ldots,\Lambda_{j-1},\Lambda_{j+1},\ldots,\Lambda_N} \max_{\Lambda_j} I(\gH_1,\ldots,\gH_{k-1} : \gH_{j} | \gH_k \ldots \gH_{j-1} \gH_{j+1} \ldots, \gH_N)_{\nu}. 
\end{multline*}
By symmetry of $\Gamma$ then 
$$ I(\gH_1,\ldots,\gH_{k-1} : \gH_{j} | \gH_{j+1} \ldots, \gH_N)_{\nu_j} \geq \min_{\Lambda_{k+1},\ldots,\Lambda_N} \max_{\Lambda_k} I(\gH_1,\ldots,\gH_{k-1} : \gH_{k} | \gH_{k+1} \ldots, \gH_N)_{\nu}$$
and since this corresponds to choosing a particular class of measurement $\cE$ 
$$ I(\gH_1,\ldots,\gH_{k-1} : \gH_{j} | \gH_{j+1} \ldots, \gH_N)_{\nu_j} \geq \min_{\cE} \max_{\Lambda_k} I(\gH_1,\ldots,\gH_{k-1} : \gH_{k} | \gH_{k+1} \ldots, \gH_N)_{\Lambda_k \otimes \cE \, \Gamma}.$$
Thus, for the particular choice of measurements we made,  
\begin{multline*}
(N-k+1) \min_{\cE} \max_{\Lambda_k}  I (\gH_1,\ldots,\gH_{k-1} : \gH_k | \gH_{k+1},\ldots , \gH_N)_{\Lambda_k \otimes \cE \, \Gamma} \\ \leq \sum_{j=k} ^N I(\gH_1,\ldots,\gH_{k-1} : \gH_{j} | \gH_{j+1} \ldots, \gH_N)_{\nu_j}
\end{multline*}
and Claim~\eqref{eq:step 2} is proved upon using~\eqref{eq:use chain}. 

\medskip 

\noindent\textbf{Conclusion.} Lemma~\ref{lem:bound info} yields 
$$ \max_{\Lambda_k, \ldots \Lambda_{N}} I(\gH_1,\ldots,\gH_{k-1} : \gH_{k}, \ldots, \gH_N)_{\nu} \leq 2 \log \mathrm{dim} (\gH^{\otimes k-1}) = 2 (k-1) \log d.$$
Insert this in~\eqref{eq:step 2}, use~\eqref{eq:step 1} and the result is proved. 
\end{proof}

Finally, we explain how Theorem~\ref{thm:BranHar} follows from Theorem~\ref{thm:deF inf}:

\begin{proof}
Denote 
$$\bL_k = \Lambda_1 \otimes \ldots \otimes \Lambda_k$$
for brevity. Using Lemma~\ref{lem:decomp mu}, convexity and then Pinsker's inequality (see~\cite{CarLie-14} or~\cite[Section~5.4]{Hayashi-06})
\begin{align}\label{eq:final}
\norm{\bL_k \left(\Gamma ^{(k)}  - \Gammat ^{(k)}\right)}^2_{\gS ^1 (\gH^k)} &= \norm{\sum_\mu p_\mu \bL_k  \left(\Gamma_\mu ^{(k)}  - \left( \Gamma_\mu ^{(1)}\right) ^{\otimes k}\right) }^2_{\gS ^1 (\gH^k)}\nonumber \\
&\leq \sum_\mu  p_\mu \norm{\bL_k \left(\Gamma_\mu ^{(k)}  - \left( \Gamma_\mu ^{(1)} \right)^{\otimes k}\right) }^2_{\gS ^1 (\gH^k)} \nonumber \\
&\leq 2 \sum_\mu  p_\mu \cH \left( \bL_k \Gamma_\mu ^{(k)}, \bL_k \left( \Gamma_\mu ^{(1)}\right) ^{\otimes k}\right) 
\end{align}
%
The last term in~\eqref{eq:final} is nothing but the multipartite mutual information refered to in Theorem~\ref{thm:deF inf}, and thus the proof is complete.
\end{proof}

%

\begin{thebibliography}{10}

\bibitem{AmmNie-08}
{\sc Z.~Ammari and F.~Nier}, {\em Mean field limit for bosons and infinite
  dimensional phase-space analysis}, Ann. Henri Poincar\'e, 9 (2008),
  pp.~1503--1574.

\bibitem{BarGolMau-00}
{\sc C.~Bardos, F.~Golse, and N.~J. Mauser}, {\em Weak coupling limit of the
  {$N$}-particle {S}chr\"odinger equation}, Methods Appl. Anal., 7 (2000),
  pp.~275--293.
\newblock Cathleen Morawetz: a great mathematician.

\bibitem{BenOliSch-12}
{\sc N.~{Benedikter}, G.~{de Oliveira}, and B.~{Schlein}}, {\em {Quantitative
  Derivation of the Gross-Pitaevskii Equation}}, Comm. Pure App. Math., 68
  (2015), pp.~1399--1482.

\bibitem{BocBreCenSch-18b}
{\sc C.~Boccato, C.~Brennecke, S.~Cenatiempo, and B.~Schlein}, {\em {Optimal
  Rate for Bose-Einstein Condensation in the Gross-Pitaevskii Regime}},
  Communications in Mathematical Physics,  (2019).

\bibitem{BraChrHarWal-16}
{\sc F.~Brand\~{a}o, M.~Christandl, A.~Harrow, and M.~Walter}, {\em {The
  Mathematics of Entanglement}}.
\newblock arXiv:1604.01790, 2016.

\bibitem{BraHar-12}
{\sc F.~Brand\~{a}o and A.~Harrow}, {\em {Quantum de Finetti Theorems under
  Local Measurements with Applications}}, Commun. Math. Phys., 353 (2017),
  pp.~469--506.

\bibitem{CarLie-14}
{\sc E.~A. Carlen and E.~H. Lieb}, {\em Remainder terms for some quantum
  entropy inequalities}, J. Math. Phys., 55 (2014), p.~042201.

\bibitem{CheHol-13}
{\sc X.~{Chen} and J.~{Holmer}}, {\em {Focusing Quantum Many-body Dynamics: The
  Rigorous Derivation of the 1D Focusing Cubic Nonlinear Schr\"odinger
  Equation}}, Arch. Rat. Mech. Anal., 221 (2016), pp.~631--676.

\bibitem{CheHol-15}
\leavevmode\vrule height 2pt depth -1.6pt width 23pt, {\em {The rigorous
  derivation of the {2D} cubic focusing {NLS} from quantum many-body
  evolution}}, Int Math Res Notices,  (2016).

\bibitem{Chiribella-11}
{\sc G.~Chiribella}, {\em On quantum estimation, quantum cloning and finite
  quantum de {F}inetti theorems}, in Theory of Quantum Computation,
  Communication, and Cryptography, vol.~6519 of Lecture Notes in Computer
  Science, Springer, 2011.

\bibitem{ChrKonMitRen-07}
{\sc M.~Christandl, R.~K{\"o}nig, G.~Mitchison, and R.~Renner}, {\em
  One-and-a-half quantum de {F}inetti theorems}, Comm. Math. Phys., 273 (2007),
  pp.~473--498.

\bibitem{ErdSchYau-07}
{\sc L.~Erd{\"{o}}s, B.~Schlein, and H.-T. Yau}, {\em Derivation of the cubic
  non-linear {S}chr\"odinger equation from quantum dynamics of many-body
  systems}, Invent. Math., 167 (2007), pp.~515--614.

\bibitem{ErdSchYau-09}
{\sc L.~Erd{\H{o}}s, B.~Schlein, and H.-T. Yau}, {\em Rigorous derivation of
  the {G}ross-{P}itaevskii equation with a large interaction potential}, J.
  Amer. Math. Soc., 22 (2009), pp.~1099--1156.

\bibitem{FroKnoSch-09}
{\sc J.~Fr{\"o}hlich, A.~Knowles, and S.~Schwarz}, {\em On the mean-field limit
  of bosons with {C}oulomb two-body interaction}, Commun. Math. Phys., 288
  (2009), pp.~1023--1059.

\bibitem{Girardot-19}
{\sc T.~Girardot}, {\em Average field approximation for almost bosonic anyons
  in a magnetic field}.
\newblock arXiv:1910.09310, 2019.

\bibitem{Harrow-13}
{\sc A.~Harrow}, {\em The church of the symmetric subspace}, preprint arXiv,
  (2013).

\bibitem{Hayashi-06}
{\sc M.~Hayashi}, {\em Quantum information}, Springer-Verlag, 2006.

\bibitem{HudMoo-75}
{\sc R.~L. Hudson and G.~R. Moody}, {\em Locally normal symmetric states and an
  analogue of de {F}inetti's theorem}, Z. Wahrscheinlichkeitstheor. und Verw.
  Gebiete, 33 (1975/76), pp.~343--351.

\bibitem{JebLeoPic-16}
{\sc M.~Jeblick, N.~Leopold, and P.~Pickl}, {\em Derivation of the time
  dependent gross-pitaevskii equation in two dimensions}.
\newblock arXiv:1608.05326, 2016.

\bibitem{KonRen-05}
{\sc R.~K\"{o}nig and R.~Renner}, {\em A de {F}inetti representation for finite
  symmetric quantum states}, J. Math. Phys., 46 (2005), p.~122108.

\bibitem{Lewin-ICMP}
{\sc M.~Lewin}, {\em {Mean-Field limit of Bose systems: rigorous results}},
  Preprint (2015) arXiv:1510.04407.

\bibitem{Lewin-11}
\leavevmode\vrule height 2pt depth -1.6pt width 23pt, {\em Geometric methods
  for nonlinear many-body quantum systems}, J. Funct. Anal., 260 (2011),
  pp.~3535--3595.

\bibitem{LewNamRou-14}
{\sc M.~Lewin, P.~Nam, and N.~Rougerie}, {\em Derivation of {H}artree's theory
  for generic mean-field {B}ose systems}, Adv. Math., 254 (2014), pp.~570--621.

\bibitem{LewNamRou-14b}
\leavevmode\vrule height 2pt depth -1.6pt width 23pt, {\em Remarks on the
  quantum de {F}inetti theorem for bosonic systems}, Appl. Math. Res. Express
  (AMRX), 2015 (2015), pp.~48--63.

\bibitem{LewNamRou-14c}
\leavevmode\vrule height 2pt depth -1.6pt width 23pt, {\em The mean-field
  approximation and the non-linear {S}chr\"odinger functional for trapped
  {B}ose gases}, Trans. Amer. Math. Soc, 368 (2016), pp.~6131--6157.

\bibitem{LewNamRou-15}
\leavevmode\vrule height 2pt depth -1.6pt width 23pt, {\em {A note on 2D
  focusing many-boson systems}}, Proc. Ame. Math. Soc., 145 (2017),
  pp.~2441--2454.

\bibitem{LewNamRou-18a}
\leavevmode\vrule height 2pt depth -1.6pt width 23pt, {\em Blow-up profile of
  rotating 2d focusing bose gases}, in Macroscopic Limits of Quantum Systems, a
  conference in honor of Herbert Spohn's 70th birthday, Springer, 2018,
  pp.~145--170.

\bibitem{LiSmi-15}
{\sc K.~Li and G.~Smith}, {\em {Quantum de Finetti Theorems under fully-one-way
  adaptative measurements}}, Phys. Rev. Lett. 114, 114 (2015), p.~160503.

\bibitem{Lieb-75}
{\sc E.~H. Lieb}, {\em {Some convexity and subadditivity properties of
  entropy}}, Bulletin of the American Mathematical Society, 81 (1975),
  pp.~444--446.

\bibitem{LieRus-73a}
{\sc E.~H. Lieb and M.~B. Ruskai}, {\em A fundamental property of
  quantum-mechanical entropy}, Phys. Rev. Lett., 30 (1973), pp.~434--436.

\bibitem{LieRus-73b}
\leavevmode\vrule height 2pt depth -1.6pt width 23pt, {\em Proof of the strong
  subadditivity of quantum-mechanical entropy}, J. Math. Phys., 14 (1973),
  pp.~1938--1941.
\newblock With an appendix by B. Simon.

\bibitem{LieSei-02}
{\sc E.~H. Lieb and R.~Seiringer}, {\em {Proof of Bose-Einstein Condensation
  for Dilute Trapped Gases}}, Phys. Rev. Lett., 88 (2002), p.~170409.

\bibitem{LieSei-06}
\leavevmode\vrule height 2pt depth -1.6pt width 23pt, {\em Derivation of the
  {G}ross-{P}itaevskii equation for rotating {B}ose gases}, Commun. Math.
  Phys., 264 (2006), pp.~505--537.

\bibitem{LieSeiSolYng-05}
{\sc E.~H. Lieb, R.~Seiringer, J.~P. Solovej, and J.~Yngvason}, {\em The
  mathematics of the {B}ose gas and its condensation}, Oberwolfach {S}eminars,
  Birkh{\"a}user, 2005.

\bibitem{LieSeiYng-00}
{\sc E.~H. Lieb, R.~Seiringer, and J.~Yngvason}, {\em Bosons in a trap: A
  rigorous derivation of the {G}ross-{P}itaevskii energy functional}, Phys.
  Rev. A, 61 (2000), p.~043602.

\bibitem{LieSeiYng-01}
\leavevmode\vrule height 2pt depth -1.6pt width 23pt, {\em A rigorous
  derivation of the {Gross-Pitaevskii} energy functional for a two-dimensional
  {Bo}se gas}, Comm. Math. Phys., 224 (2001), pp.~17--31.

\bibitem{LieYng-98}
{\sc E.~H. Lieb and J.~Yngvason}, {\em Ground state energy of the low density
  {B}ose gas}, Phys. Rev. Lett., 80 (1998), pp.~2504--2507.

\bibitem{LieYng-01}
\leavevmode\vrule height 2pt depth -1.6pt width 23pt, {\em The ground state
  energy of a dilute two-dimensional {B}ose gas}, J. Stat. Phys., 103 (2001),
  p.~509.

\bibitem{LunRou-15}
{\sc D.~Lundholm and N.~Rougerie}, {\em {The average field approximation for
  almost bosonic extended anyons}}, J. Stat. Phys., 161 (2015), pp.~1236--1267.

\bibitem{MulRee-17}
{\sc A.~Müller-Hermes and D.~Reeb}, {\em Monotonicity of the quantum relative
  entropy under positive maps}, Annales Henri Poincar\'e, 18 (2017),
  pp.~1777--1788.

\bibitem{NamNap-17}
{\sc P.~Nam and M.~Napi\'orkowski}, {\em {Norm approximation for many-body
  quantum dynamics: focusing case in low dimensions}}.
\newblock arXiv:1710.09684, 2017.

\bibitem{NamRou-19}
{\sc P.~Nam and N.~Rougerie}, {\em {Improved stability for 2D attractive Bose
  gases}}.
\newblock arXiv:1909.08902, 2019.

\bibitem{NamRouSei-15}
{\sc P.~T. Nam, N.~Rougerie, and R.~Seiringer}, {\em {Ground states of large
  Bose systems: The Gross-Pitaevskii limit revisited}}, Analysis and PDEs, 9
  (2016), pp.~459--485.

\bibitem{OhyPet-93}
{\sc M.~Ohya and D.~Petz}, {\em Quantum entropy and its use}, Texts and
  Monographs in Physics, Springer-Verlag, Berlin, 1993.

\bibitem{Pickl-10}
{\sc P.~Pickl}, {\em Derivation of the time dependent {G}ross-{P}itaevskii
  equation without positivity condition on the interaction}, J. Stat. Phys.,
  140 (2010), pp.~76--89.

\bibitem{Pickl-11}
\leavevmode\vrule height 2pt depth -1.6pt width 23pt, {\em A simple derivation
  of mean-field limits for quantum systems}, Lett. Math. Phys., 97 (2011),
  pp.~151--164.

\bibitem{RodSch-09}
{\sc I.~Rodnianski and B.~Schlein}, {\em Quantum fluctuations and rate of
  convergence towards mean field dynamics}, Commun. Math. Phys., 291 (2009),
  pp.~31--61.

\bibitem{Rougerie-LMU}
{\sc N.~Rougerie}, {\em {De Finetti theorems, mean-field limits and
  Bose-Einstein condensation}}.
\newblock arXiv:1506.05263, 2014.
\newblock LMU lecture notes.

\bibitem{Rougerie-spartacus}
\leavevmode\vrule height 2pt depth -1.6pt width 23pt, {\em Th{\'e}or{\`e}mes de
  De Finetti, limites de champ moyen et condensation de Bose-Einstein}, Les
  cours Peccot, Spartacus IDH, Paris, 2016.
\newblock Cours Peccot, Coll{\`e}ge de France : f{\'e}vrier-mars 2014.

\bibitem{Rougerie-EMS}
\leavevmode\vrule height 2pt depth -1.6pt width 23pt, {\em {Scaling limits of
  bosonic ground states}}.
\newblock arXiv, 2019.
\newblock In preparation.

\bibitem{Schatten-60}
{\sc R.~Schatten}, {\em Norm Ideals of Completely Continuous Operators}, vol.~2
  of Ergebnisse der Mathematik und ihrer Grenzgebiete, Folge, 1960.

\bibitem{Seiringer-02}
{\sc R.~Seiringer}, {\em Gross-{P}itaevskii theory of the rotating {B}ose gas},
  Commun. Math. Phys., 229 (2002), pp.~491--509.

\bibitem{Seiringer-03}
\leavevmode\vrule height 2pt depth -1.6pt width 23pt, {\em Ground state
  asymptotics of a dilute, rotating gas}, J. Phys. A, 36 (2003),
  pp.~9755--9778.

\bibitem{Simon-79}
{\sc B.~Simon}, {\em Trace ideals and their applications}, vol.~35 of London
  Mathematical Society Lecture Note Series, Cambridge University Press,
  Cambridge, 1979.

\bibitem{Triay-17}
{\sc A.~Triay}, {\em {Derivation of the dipolar Gross--Pitaevskii energy}},
  SIAM J. Math. Anal., 50 (2018), pp.~33--63.

\end{thebibliography}

\end{document}